\pdfoutput=1
\documentclass[10pt,twoside]{article}

\usepackage{lmodern}
\usepackage[T1]{fontenc}
\usepackage[utf8]{inputenc}
\usepackage{textcomp}
\usepackage[english]{babel}

\usepackage{amsmath}
\usepackage{amssymb}
\usepackage{amsthm}
\usepackage{mathtools}
\usepackage{bigstrut}
\usepackage{cases}
\usepackage{graphicx}
\usepackage{tabularx}
\usepackage{enumerate}
\usepackage{comment}
\usepackage{mathrsfs}
\usepackage[linesnumbered,ruled]{algorithm2e}
\usepackage{bm}
\usepackage{multirow}

\usepackage[top=31mm, bottom=31mm, left=33mm, right=33mm]{geometry}

\usepackage{fancyhdr}
\newcommand{\R}{\mathbf{R}}

\DeclareMathOperator*{\chol}{chol}

\DeclareMathOperator*{\argmin}{arg\,min}
\DeclareMathOperator*{\argmax}{arg\,max}
\DeclareMathOperator*{\essinf}{ess\,inf}
\DeclareMathOperator*{\dist}{dist}

\newtheorem*{lemma}{Lemma}

\mathtoolsset{showonlyrefs,showmanualtags}

\usepackage[pdftex,colorlinks,pdftitle={Thermal tomography with unknown boundary}]{hyperref}
\hypersetup{linkcolor=black}
\hypersetup{urlcolor=black}
\hypersetup{citecolor=black}

\begin{document}

\title{Thermal tomography with unknown boundary}
\author{Nuutti Hyvönen\footnote{Aalto University, Department of Mathematics and Systems Analysis, P.O.\ Box 11100, FI-00076 Aalto, Finland (nuutti.hyvonen@aalto.fi).} \and Lauri Mustonen\footnote{Aalto University, Department of Mathematics and Systems Analysis, P.O.\ Box 11100, FI-00076 Aalto, Finland (lauri.mustonen@aalto.fi).}}

\maketitle

\begin{abstract}
Thermal tomography is an imaging technique for deducing information about the internal structure of a physical body from temperature measurements on its boundary. This work considers time-dependent thermal tomography modeled by a parabolic initial/boundary value problem without accurate information on the  exterior shape of the examined object. The adaptive sparse pseudospectral approximation method is used to form a polynomial surrogate for the dependence of the temperature measurements on the thermal conductivity, the heat capacity, the boundary heat transfer coefficient and the body shape. These quantities can then be efficiently reconstructed via nonlinear, regularized least squares minimization employing the surrogate and its derivatives. The functionality of the resulting reconstruction algorithm is demonstrated by numerical experiments based on simulated data in two spatial dimensions.

\bigskip


\end{abstract}

\section{Introduction}
\label{sec:intro}

The aim of thermal tomography is to gather information about the interior properties of a physical object based on temperature measurements on its boundary.
The potential applications include nondestructive testing of materials; see,~e.g.,~\cite{bakirov05,toivanen12} for further introduction to the topic as well as for a list of previous works.
The model used in this article is time-dependent: the boundary temperatures are measured several times while heating the object.
Thermal conductivity and heat capacity are both considered as unknown quantities as in \cite{kolehmainen08}.
In addition, the boundary heat transfer coefficient between the heater elements is reconstructed.
Such a setting was considered with simulated data in \cite{toivanen12} and subsequently with experimental data in \cite{toivanen14}.
As a novelty of this work, the exterior shape of the imaged object is also reconstructed.

In this article, we approximate the dependence of the temperature measurements on the four unknown quantities by multivariate polynomials.
The unknown functions (i.e.,~the thermal conductivity, heat capacity, boundary heat transfer coefficient and boundary shape) are first expressed in terms of a finite number of parameters.
The polynomial forward surrogate, which takes these parameters as input, is then constructed by the (adaptive) {\em sparse pseudospectral approximation method} (SPAM) \cite{constantine12,conrad13,winokur13,winokur15} that requires the solutions to a large number of parabolic initial/boundary value problems describing the heat flow in the examined object.
When a measurement vector is given, the inverse problem of determining the unknowns is treated as nonlinear, regularized least squares minimization involving the forward surrogate.
The advantage of this approach is that only polynomial evaluation and differentiation are needed when reconstructing the quantities of interest from the measurement data, which leads to a fast reconstruction process.
More precisely, under the assumption that enough generic information about the measurement setup is available, the polynomial forward surrogate can be formed `offline', that is, before the actual measurements are in hand.
Thus, the efficiency of the `online' phase, or the least squares reconstruction algorithm, does not depend on how the parabolic forward problems were solved in the offline phase.
For a related polynomial approximation approach to thermal tomography assuming the shape of the imaged body is known and without modeling the effects of thermal capacitance or the geometry of the heaters and sensors, see~\cite{mustonen15}, where a (non-adaptive) Galerkin technique is used for forming the needed surrogates. To the best of our knowledge, apart from~\cite{mustonen15}, there are no previous works considering polynomial surrogates as a computational tool in the framework of inverse parabolic boundary value problems.

In practical applications, the exterior shape of the imaged body is usually not the object of main interest, and the same definitely applies to the boundary heat transfer coefficient.
Typically, the interesting quantities are the thermal conductivity and the heat capacity, or perhaps only one of those,~e.g.,~if there exists a known functional dependence between the two of them.
However, we demonstrate that ignoring the uncertainties in the object shape and the boundary heat transfer coefficient deteriorates the reconstructions of the thermal conductivity and the heat capacity and therefore we suggest that all four unknowns should be reconstructed simultaneously if possible.
Although the main reason for only considering two-dimensional examples is their computational simplicity, there exists also some practical relevance since a two-dimensional model can be utilized if the object is cylindrically symmetric with suitable boundary conditions at the top and bottom boundaries.

A related shape estimation problem has previously been studied in the context of {\em electrical impedance tomography} (EIT) in~\cite{darde13a,darde13b,hyvonen16,nissinen11}. The fundamental difference between EIT and thermal tomography is that the former is modeled by an elliptic partial differential equation with one unknown physical field (conductivity), whereas thermal tomography corresponds to an (inverse) parabolic boundary value problem with two coefficients defining the properties of the imaged body (thermal conductivity and heat capacity). The algorithm of \cite{nissinen11} is based on the so called approximation error approach \cite{Kaipio04a}: the error caused by mismodeling the measurement setup is treated as an additional noise process whose statistics are approximated via simulations based on prior probability models for the unknowns. The papers \cite{darde13a,darde13b} introduce Fr\'echet derivatives of the forward map of EIT with respect to the measurement geometry, which allows one to estimate the body shape as a part of a Bayesian output least squares algorithm, although the evaluation of the needed shape derivatives suffers from slight instability that originates from certain distributional boundary conditions in the elliptic boundary value problems defining the derivatives. Our approach is closely related to \cite{hyvonen16} where a polynomial surrogate is (non-adaptively) built for the forward operator of EIT. However, the efficiency gain induced by transferring computational burden to an offline phase is higher in thermal tomography than in EIT due to the possibility to completely avoid time integration in the online phase when the reconstruction is formed by minimizing a least squares functional.  

We denote integer vectors, or \emph{multi-indices}, by bold symbols such as $\boldsymbol{k} \coloneqq (k_1, \ldots, k_N)$.
In particular, the $n$th Euclidean basis vector is denoted by $\boldsymbol{e}_n$, having $1$ as the $n$th element and $0$ in the remaining $N-1$ positions, where $N$ is assumed to be clear from the context.
The calligraphic font, e.g.~$\mathcal{K}$, is used for sets of multi-indices.
Linear operators are denoted by scripted letters such as $\mathscr{Q}$ or $\mathscr{P}$.
Symbols $\mathbf{R} = (-\infty, \infty)$ and $\mathbf{N} = \{1,2,3,\ldots\}$ as well as $\mathbf{N}_0 = \mathbf{N} \cup \{0\}$ have their usual meanings.
Furthermore, $\mathbf{P}_k$ denotes the space of real-valued polynomials in one real variable, having a degree $k \in \mathbf{N}_0$ or less.

This text is organized as follows.
In Section~\ref{sec:tomo}, the governing equations for heat flow are introduced and the dependence of the measurements on the finite-dimensional parameter vector is formulated.
The pseudospectral approximation method and its adaptive version are described in Section~\ref{sec:pseudospectral}.
Section~\ref{sec:num} presents the numerical experiments that address both the accuracy of the forward surrogate and the feasibility of thermal tomography with an unknown exterior boundary.
Finally, conclusions are drawn in Section~\ref{sec:conclusions}.

\section{Thermal tomography}
\label{sec:tomo}

We start this section by mathematically formulating our forward model for thermal tomography following~\cite{toivanen12,toivanen14}. Subsequently, a parametric extension is introduced for a certain two-dimensional measurement configuration.

\subsection{Measurement model}
\label{ssec:forward}

Let $\varOmega \subset \mathbf{R}^d$, where $d \in \{2,3\}$, denote a bounded domain with a $C^1$-boundary $\partial \varOmega$ and the exterior unit normal $\hat{\boldsymbol{n}}$.
Furthermore, let $I = (0,T)$ denote a time interval with $T>0$.
The initial/boundary value problem considered as a model for the temperature $u$ in thermal tomography is
  \begin{equation}
  \label{eq:strong}
    \begin{dcases}
    b\, \partial_t u - \nabla \cdot (a \nabla u) = 0 & \text{in } \varOmega \times I, \\
    a \nabla u \cdot \hat{\boldsymbol{n}} = c \,(f-u) & \text{on } \partial \varOmega \times I, \\
    u = 0 & \text{in } \varOmega \times \{0\},
    \end{dcases}
  \end{equation}
where $a, b \in L_+^\infty(\varOmega)$ are the thermal conductivity and the heat capacity, respectively, and $c \in L_+^\infty(\partial \varOmega)$ is the surface conductance or heat transfer coefficient. The time-dependent boundary source $f \in L^2(I; L^2(\partial \varOmega))$ represents the external or outside temperature. All aforementioned function spaces are assumed to have $\mathbf{R}$ as the corresponding multiplier field and
$$
L^\infty_+(D) \, := \, \{ v \in L^\infty(D) \, : \, {\rm ess} \inf v > 0 \}
$$ 
for $D = \varOmega$ or $\partial \varOmega$. The weak formulation of \eqref{eq:strong} is to find $u \colon \varOmega \times I \to \mathbf{R}$ such that
  \begin{equation}
  \label{eq:weak}
  \partial_t (b \, u, v)_{\varOmega} + (a \nabla u, \nabla v)_{\varOmega} + ( c\, u, v )_{\partial \varOmega} = ( c \, f, v )_{\partial \varOmega} \qquad {\rm for} \ {\rm all} \ v \in H^1(\varOmega),
  \end{equation}
where $(\cdot, \cdot)_D$ denotes the (real) $L^2$ inner product on $D$ and the time derivative as well as the initial condition from \eqref{eq:strong} are to be understood in the appropriate distributional sense (cf.,~e.g.,~\cite{dautray92}). Under the above assumptions, the variational problem \eqref{eq:weak} has a unique solution $u \in V$, where
  \begin{equation}
   V \coloneqq L^2(I; H^1(\varOmega)) \cap C(\overline{I}; L^2(\varOmega));
  \end{equation}
see,~e.g.,~\cite{dautray92}.

We assume that the boundary is equipped with $J \in \mathbf{N}$ non-overlapping heater elements $\{H_j\}_{j=1}^J$ that are identified with the open, nonempty and connected parts of $\partial \varOmega$ that they cover. In particular, it is natural to model the boundary heat source as
$$
f \, = \, \sum_{j=1}^J g_j \chi_j ,
$$
where $g_j \in L^2(I)$ and $\chi_j \colon \partial \varOmega \to \{0,1\}$ is the characteristic function of $H_j$. Moreover, since our model for thermal tomography involves heating only one heater at a time and letting the target object to cool down before activating the next heater, we actually consider solving \eqref{eq:strong} for a relatively short period of time but for $J$ different functions
\begin{equation}
\label{eq:bsource}
f = f_j  :=  g \, \chi_j, \qquad j=1, \dots, J,
\end{equation}
where $g \in L^2(I)$ is the chosen common time modulation for the heating patterns. Usually, the heat transfer coefficient $c$ is large on those parts of the boundary that are covered by heaters and much smaller (but still positive) on the rest of the boundary (cf.~\cite{toivanen14}).
Note that even though $f$ vanishes outside the heaters, the heat still leaks to the surrounding air due to the non-vanishing Robin coefficient $c$.

It is assumed that the boundary temperature of $\varOmega$ is measured at $R$ point-like sensors $\{s_r\}_{r=1}^R \subset \partial \varOmega$ that are well-separated from the heaters. We denote by $U \in \mathbf{R}^M$ a measurement vector that consists of the temperature values at the sensor locations for each measurement time $t_i \subset \overline{I}$, $i=1,\dots, M_T$. (To make such measurements well defined without further regularity assumptions for~\eqref{eq:strong}, a single measurement should be interpreted as the mean value of $u|_{\partial \varOmega \times I} \in L^2(I; H^{1/2}(\partial \varOmega)) \subset L^2(I; L^2(\partial \varOmega)) \simeq L^2(\partial \varOmega \times I)$ over a small neighborhood of the corresponding point $(s_r, t_i)$ on $\partial \varOmega \times \overline{I}$.) Assuming that for each active heater,~i.e.,~for every $f_j$, all $R$ sensor values are recorded at all measurement times, the total number of measurements becomes $M = JRM_T$.

After a {\em finite element} (FE) discretization, \eqref{eq:weak} results in a system of ordinary differential equations
  \begin{equation}
  \label{eq:ode}
  \partial_t B \, \widehat{u} + (A + C)\, \widehat{u} = \widehat{f},
  \end{equation}
where $A$, $B$ and $C$ are symmetric matrices corresponding to the three integrals on the left-hand side of \eqref{eq:weak}. Moreover, $\widehat{u}$ and $\widehat{f}$ are vectors carrying the coefficients of $u$ and $cf$ in the chosen FE basis. In practice, $\widehat{f}$ and $\widehat{u}$ can also be matrices whose $J$ columns correspond to heating different heaters.
The semi-discrete system \eqref{eq:ode} is solved in time by using the second-order Crank--Nicolson time integration method \cite{Larson13}.
The measurement vector $U$ is computed by evaluating the numerical solution at the specified measurement times and locations.

\subsection{Parametric extension}
\label{ssec:param}

In order to build a surrogate model that describes the dependence of the measurements on the fields $a$ and $b$ as well as on the coefficient $c$ and the boundary shape $\partial \varOmega$, we parametrize these quantities with a finite number of parameters.
More precisely, after denoting by $N = N_a + N_b + N_c + N_\varOmega$ the total number of parameters, we consider parameter vectors $\vartheta = (\vartheta^{(a)}, \vartheta^{(b)}, \vartheta^{(c)}, \vartheta^{(\varOmega)}) \in \varXi^N$, where $\varXi = [-1/2, 1/2]$ and $\varXi^N$ is the parameter hypercube with unit volume. The idea is that each fixed parameter vector defines one well posed initial/boundary value problem of the form \eqref{eq:strong}. That is, for each parameter vector there exists a unique solution $u$ and a unique measurement vector $U$ so that these quantities can be interpreted as mappings $u \colon \varXi^N \to V$ and $U \colon \varXi^N \to \mathbf{R}^M$, respectively. In the following, we only consider a specific two-dimensional parametrization, but emphasize that the extension to more complicated geometries or to three dimensions is conceptually straightforward. 

Let us begin by considering a suitable representation for the shape of a domain $\varOmega \subset \R^2$ as a perturbation of a disk. To this end, the boundary curve is parametrized in polar coordinates by two variables, the polar angle $\phi$ and the shape parameter $\vartheta^{(\varOmega)}\in \varXi^{N_\varOmega}$, that is,
\begin{equation}
\label{eq:boundparam}
\partial \varOmega (\vartheta^{(\varOmega)}) \, = \, \big\{ \big( r(\phi;  \vartheta^{(\varOmega)}), \phi \big) \ : \ \phi \in [0, 2\pi) \big\}.
\end{equation}
The radial coordinate is chosen to be of the form
\begin{equation}
\label{radial_coor}
r(\phi; \vartheta^{(\varOmega)}) = \rho_0 + \sum_{i=1}^{N_\varOmega} (\rho_+ - \rho_-) \, \vartheta^{(\varOmega)}_i  \psi_i(\phi),
\end{equation}
where $\rho_0 = (\rho_+ - \rho_-)/2$ and $\rho_+ > \rho_-$ are the maximal and minimal radii of the parametrized domain, respectively. The local perturbations $\psi_i \in C^1([0,2\pi))$ are $2 \pi$-periodic, nonnegative and uniform quadratic B-splines that form a partition of unity on the unit circle. We also define a homeomorphism $\Phi(\, \cdot \, ; \vartheta^{(\varOmega)}) \colon \overline{\varOmega(\vartheta^{(\varOmega)})} \to \overline{\varOmega(0)}$ from a given perturbed domain to the unperturbed reference disk of radius $\rho_0$ via
\begin{equation}
\label{homeodef}
\Phi \big((r',\phi);  \vartheta^{(\varOmega)} \big) = \bigg(\frac{\rho_0}{r(\phi,\vartheta^{(\varOmega)})} \, r', \phi \bigg)
\end{equation}
for $(r', \phi) \in \overline{\varOmega(\vartheta^{(\varOmega)})}$. It is assumed that the widths and starting polar angles (in the counter clockwise direction) of the heaters are fixed (i.e.,~known) and that the same applies to the polar angles of the point-like heat sensors.\footnote{It is natural to assume that the widths of the heaters are known, but the same does not necessarily apply to the (starting) polar angles of the heaters and the sensors. However, the precise nature of the available information on the positioning of the heaters and sensors is application-specific, and thus we have decided to work with this relatively general and simple assumption.} In other words, given the starting points $\{y_j(0) \}_{j=1}^J$ and the common width $\eta$ of the heater patches $\{H_j(0)\}_{j=1}^J$ as well as the sensor locations $\{s_r(0)\}_{r=1}^R$ on the boundary of the reference disk $\varOmega(0)$, the dependence of these geometric entities on the shape parameter $\vartheta^{(\varOmega)}$ can be written as
\begin{align*}
H_j\big(\vartheta^{(\varOmega)}\big) & = \big\{ x \in \partial \varOmega(\vartheta^{(\varOmega)}) \ : \ 0 < \dist\!\big(\Phi^{-1}(y_j(0); \vartheta^{(\varOmega)}),x \big) < \eta \big\}, \\[2mm] 
s_r(\vartheta^{(\varOmega)}) &= \Phi^{-1}\big(s_r(0);\vartheta^{(\varOmega)}\big),
\end{align*}
where $\dist(y,x) = \dist(y,x;\vartheta^{(\varOmega)})$ denotes the distance between the points $y,x \in \partial \varOmega(\vartheta^{(\varOmega)})$ along the boundary in the counter clockwise direction.
In particular, the boundary heat sources are still modeled by~\eqref{eq:bsource} under the interpretation that $\chi_j = \chi_j(\,\cdot \,; \vartheta^{(\varOmega)})$ is the characteristic function of $H_j(\vartheta^{(\varOmega)})$ for some $\vartheta^{(\varOmega)} \in \varXi^{N_\varOmega}$ that is clear from the context. 

The parameter-dependent function-valued mappings $a,b \colon \varXi^N \to L_+^\infty(\varOmega(\vartheta^{(\varOmega)}))$ are defined through
  \begin{align}
  \label{eq:abparam}
  \begin{split}
  a(x; \vartheta) &= \bar{a}\big(\Phi(x; \vartheta^{(\varOmega)})\big) + 2  \, \widetilde{a} \, \sum_{i=1}^{N_a} \vartheta^{(a)}_i \varphi^{(a)}_i\!\big(\Phi(x; \vartheta^{(\varOmega)})\big), \\ 
  b(x; \vartheta) &= \bar{b}\big(\Phi(x; \vartheta^{(\varOmega)})\big) + 2  \, \widetilde{b} \, \sum_{i=1}^{N_b} \vartheta^{(b)}_i \varphi^{(b)}_i\!\big(\Phi(x; \vartheta^{(\varOmega)})\big),
  \end{split}
  \end{align}
where $\bar{a},\bar{b} \in L_+^\infty(\varOmega(0))$ are the `default fields' and the bases $\{ \varphi^{(a)}_i \}_{i=1}^{N_a}, \{ \varphi^{(b)}_i \}_{i=1}^{N_b} \in L_+^\infty(\varOmega(0))$ are assumed to form {\em partitions of unity} on $\varOmega(0)$. The free parameters $\widetilde{a}, \widetilde{b} \in \R_+$ are to be tuned so that $a(\, \cdot \,; \vartheta)$ and $b(\, \cdot \,; \vartheta)$ really are elements of $L^\infty_+(\varOmega(\vartheta^{(\varOmega)}))$ for all $\vartheta \in \varXi^N$, that is, by construction they should satisfy
$$
\widetilde{a} < \essinf \bar{a} \qquad {\rm and}  \qquad \widetilde{b} < \essinf \bar{b}.
$$
The bases $\{ \varphi^{(a)}_i \}_{i=1}^{N_a}$ and $\{ \varphi^{(b)}_i \}_{i=1}^{N_b}$ could be formed using,~e.g.,~B-splines or piecewise linear FE basis, but in the numerical experiments of Section~\ref{sec:num} we simply resort to characteristic functions corresponding to pixelwise discretizations of the reference disk $\varOmega(0)$.

The surface conductance could be parametrized in the same manner as $a$ and $b$. However, as the (constant) value of the conductance underneath the heaters may be assumed known (cf.~\cite{toivanen14}) and its value between the heaters is not expected to vary considerably, we employ a simpler parametrization for $c \colon \varXi^N \to L_+^\infty(\partial \varOmega(\vartheta^{(\varOmega)}))$, that is,
\begin{equation}
\label{eq:cparam}
c(x; \vartheta) = 
 \begin{cases}
    c_0 \quad &\text{if } x \in \bigcup_{j=1}^J H_j(\vartheta^{(\varOmega)}),  \\[1mm]
    \bar{c} + 2\,\widetilde{c}\,\sum_{j=1}^{J} \vartheta^{(c)}_j \tilde{\chi}_j \quad &\text{otherwise}.
    \end{cases}
\end{equation}
Here $c_0 \in \R_+$ is the known (high) value of the surface conductance at the heaters, $\bar{c}$ is its default (low) value between the heaters and $\{\tilde{\chi}_j\}_{j=1}^J$ are the characteristic functions of the $J$ curve segments between the heaters. In particular, take note that $N_c = J$. To guarantee the positivity of $c$, it is assumed that $\bar{c} > \widetilde{c}$.

We collect the assumptions on our parametrization in the following lemma.

\begin{lemma}
If the object boundary and the coefficients in \eqref{eq:strong} are parametrized according to \eqref{eq:boundparam}, \eqref{eq:abparam} and \eqref{eq:cparam} under the associated assumptions, then \eqref{eq:weak} has a unique solution $u(\, \cdot \,; \vartheta) \in V$ for all $\vartheta \in \varXi^N$. In other words, the mapping 
$\varXi^N \ni \vartheta \mapsto u(\, \cdot \,; \vartheta) \in V$ is well defined.
\end{lemma}

\begin{proof}
It is easy to see that for any fixed $\vartheta \in \varXi^N$, the equations \eqref{eq:boundparam}, \eqref{eq:abparam} and \eqref{eq:cparam} define a $C^1$-domain $\varOmega(\vartheta^{\varOmega})$ as well as coefficient functions $a(\, \cdot \,; \vartheta), b(\, \cdot \,; \vartheta) \in L^\infty_+(\varOmega(\vartheta^{\varOmega}))$ and $c(\, \cdot \,; \vartheta) \in L^\infty_+(\partial \varOmega(\vartheta^{\varOmega}))$.  The assertion thus follows from the standard theory on parabolic initial/boundary value problems; see,~e.g.,~\cite{dautray92}. 
\end{proof}

The parametrizations in \eqref{eq:abparam} and \eqref{eq:cparam} could easily be made more general; see,~e.g.,~\cite{hyvonen16}, where a logarithmic parametrization for coefficient functions is used. Moreover, dropping the requirement on the bases for the partition of unity property, the most common parametrization for a (random) coefficient in a partial differential equation is arguably the (exponential) Karhunen--Lo\`eve expansion (cf.,~e.g.,~\cite{schwab11}). In any case, it is imperative to choose the discretizations so that for each $\vartheta \in \varXi^N$ the functions $a$, $b$ and $c$ stay positive.

\section{Pseudospectral approximation}
\label{sec:pseudospectral}

Let us assume that an accurate numerical forward solver for the parametrized measurement $U \colon \varXi^N \to \mathbf{R}^M$ is available.
In other words, for any given vector $\vartheta \in \varXi^N = [-1/2, 1/2]^N$, we are able to compute the boundary measurements $\{U_m(\vartheta)\}_{m=1}^M$ according to the parametrized measurement setup and equations presented in Section~\ref{sec:tomo}.
The inaccuracies of the forward solver are neglected for the moment, and therefore the same notation $U$ is used for both the computed forward solution and the exact temperatures.
In this section, our aim is to construct a polynomial surrogate $U^{(\mathcal{K})} \colon \varXi^N \to \mathbf{R}^M$ that satisfies $U^\mathcal{(K)} \approx U$ in some appropriate sense.
The approximation parameter $\mathcal{K} \subset \mathbf{N}_0^N$ is a multi-index set that is related to projection orders and polynomial degrees as explained shortly.
The underlying motivation for using a polynomial surrogate is the possibility of faster evaluation of the forward solution and its derivatives when solving the inverse problem of thermal tomography by an iterative scheme.
Moreover, the employment of a parametric model simplifies the handling of the unknown boundary shape (cf.~\cite{darde13a,darde13b,hyvonen16}).

The construction of the surrogate is based on the SPAM introduced in \cite{constantine12} and, in particular, its adaptive version analyzed in \cite{conrad13}.
We briefly outline the algorithm in what follows and refer the interested readers to, e.g., \cite{gerstner03,winokur13,winokur15} for more details.
For simplicity of exposition, in this section, a common approximation parameter $\mathcal{K}$ and the same set of basis polynomials are assumed for every measurement component.
Alternatively, the surrogate could be constructed separately for each component of $U$.
The latter is equivalent to considering $M$ real-valued mappings $U_m$, each in turn, instead of one vector-valued mapping $U$.
Thus, the case of component-wise approximations can be treated as a special case of the procedure presented here.

\subsection{Sparse pseudospectral approximation method (SPAM)}
\label{ssec:spam}

The polynomial surrogate $U^{(\mathcal{K})}$ is written as linear combinations of orthonormal basis polynomials.
More precisely,
  \begin{equation}
  \label{eq:surrogate}
  U_m^{(\mathcal{K})}(\vartheta) \, = \sum_{\boldsymbol{i} \in \mathcal{P}(\mathcal{K}) } \alpha_{m,\boldsymbol{i}} L_{\boldsymbol{i}}(\vartheta),
  \qquad 1 \leq m \leq M, \quad \vartheta \in \varXi^N,
  \end{equation}
where $\alpha_{m,\boldsymbol{i}} \in \mathbf{R}$ are the expansion coefficients and $L_{\boldsymbol{i}}$ are suitably transformed multivariate Legendre polynomials with degrees defined by the multi-index set $\mathcal{P}(\mathcal{K}) \subset \mathbf{N}_0^N$ that depends on the approximation parameter via the chosen quadrature rule.
The basis polynomials satisfy
  \begin{equation}
  \label{eq:multiorth}
  \int_{\varXi^N} L_{\boldsymbol{i}}(\vartheta) L_{\boldsymbol{j}}(\vartheta) \,\mathrm{d} \vartheta \, = \,
    \begin{cases}
    1 \quad \text{if } i_n=j_n, \ 1 \leq n \leq N, \\
    0 \quad \text{otherwise}
    \end{cases}
  \end{equation}
for all $\boldsymbol{i}, \boldsymbol{j} \in \mathcal{P}(\mathcal{K})$
and they can be naturally constructed as products of univariate polynomials, i.e.,
  \begin{equation}
  L_{\boldsymbol{i}}(\vartheta) = \prod_{n=1}^N \ell_{i_n}(\vartheta_n),
  \end{equation}
where $\ell_{i}$ denotes the $i$th order orthonormal(ized) Legendre polynomial on $\varXi = [-1/2, 1/2]$.
We refer to \cite{gautschi04} for more details about orthogonal polynomials.
Hereafter, the subscript in $\vartheta_n$ denotes the dimension (i.e., $\vartheta_n$ is the $n$th element of the vector $\vartheta$), whereas superscripts are used to distinguish between different multidimensional nodes, i.e., vectors in $\varXi^N$.

The sparse pseudospectral approximation is built on sequential \emph{full tensor projections}.
Let us first discuss univariate quadratures and projections.
A general quadrature rule of order $k \geq0$ is a linear operator 
$\mathscr{Q}_k \colon C(\varXi) \to \mathbf{R}$, defined as\footnote{A quadrature can naturally be applied to any function having well defined values at the corresponding nodes, not only to continuous ones. The same is true for the associated projections.}
  \begin{equation}
  \label{eq:uniquad}
  \mathscr{Q}_k(F) \coloneqq \sum_{i=0}^{q(k)} F\big(\vartheta^{(k,i)}\big) w^{(k,i)}
  \approx \int_\varXi F(\vartheta) \,\mathrm{d}\vartheta,
  \end{equation}
where $q(k)+1$ is the number of quadrature nodes and $\{(\vartheta^{(i,k)}, w^{(i,k)})\}_{i=0}^{q(k)} \subset \varXi \times \mathbf{R}$ are the nodes and weights.
The rule is said to have an exactness of degree $s(k)$ if the approximation in \eqref{eq:uniquad} is an equality for all polynomials $F \in \mathbf{P}_{s(k)}$.
The function $s \colon \mathbf{N}_0 \to \mathbf{N}_0$ is assumed to be nondecreasing.
For $k \geq 0$, the $k$th order \emph{pseudospectral projection} onto the Legendre basis is $\mathscr{P}_k \colon C(\varXi) \to \mathbf{P}_{p(k)}$,
  \begin{equation}
  \mathscr{P}_k(F) (\vartheta) \coloneqq \sum_{i=0}^{p(k)} \widehat{F}_i \ell_i(\vartheta), \qquad \widehat{F}_i = \mathscr{Q}_k(F \ell_i),
  \end{equation}
where $p(k) = \lfloor s(k) / 2 \rfloor$. It is easy to verify that $\mathscr{P}_k$ really is a projection, that is, an identity map (i.e., exact) on $\mathbf{P}_{p(k)}$.
Finally, the \emph{difference projection} operator for $k \geq 1$ is $\mathscr{D}_k \colon C(\varXi) \to \mathbf{P}_{p(k)}$,
  \begin{equation}
  \mathscr{D}_k \coloneqq \mathscr{P}_k - \mathscr{P}_{k-1}
  \end{equation}
and for $k=0$ we set $\mathscr{D}_0 = \mathscr{P}_0$. Notice that $\mathscr{D}_k$ is not actually a projection itself, but only a difference of two (unless $k=0$).

The Gauss--Legendre quadrature admits the values $q(k) = k$ and $s(k) = 2k+1$, yielding the convenient property $p(k) = k$.
In contrast, the popular nested Clenshaw--Curtis rule grows exponentially, so that increasing the projection order by one gives many new polynomials.
See, e.g., \cite{conrad13} for more details about different quadrature rules, including the Gauss--Patterson rule.

Before introducing the multidimensional counterparts of the linear operators presented above, we need to fix some notation.
Firstly, $\mathbf{P}_{\boldsymbol{k}} \coloneqq (\mathbf{P}_{k_1} \otimes \cdots \otimes \mathbf{P}_{k_N})$ denotes the space of $N$-variate polynomials with the given maximum univariate degrees $k_1, \ldots, k_N$.
Secondly, the vectorization $\boldsymbol{y} \colon \mathbf{N}_0^N \to \mathbf{N}_0^N$ of a function $y \colon \mathbf{N}_0 \to \mathbf{N}_0$ is defined via $\boldsymbol{y}(\boldsymbol{k}) \coloneqq (y(k_1), \ldots, y(k_N))$.
Lastly, for a multi-index $\boldsymbol{k} \in \mathbf{N}_0^N$, the induced full tensor index set is
  \begin{equation}
  \label{eq:ftset}
  \mathcal{T}(\boldsymbol{k}) \coloneqq \{ \boldsymbol{i} \in \mathbf{N}_0^N : 0 \leq i_n \leq k_n, \; n=1,\ldots,N \},
  \end{equation}
which has the cardinality
  \begin{equation}
  \label{eq:ftcard}
  \lvert \mathcal{T}(\boldsymbol{k}) \rvert = \prod_{n=1}^N (k_n+1).
  \end{equation}

A full tensor quadrature $\mathscr{Q}_{\boldsymbol{k}} \colon C(\varXi^N) \to \mathbf{R}$ of order $\boldsymbol{k} \in \mathbf{N}_0^N$ is
  \begin{align}
  \mathscr{Q}_{\boldsymbol{k}}(F) &\coloneqq (\mathscr{Q}_{k_1} \otimes \cdots \otimes \mathscr{Q}_{k_N})(F) \\
  &=\sum_{i_1=0}^{q(k_1)} \cdots \sum_{i_N=0}^{q(k_N)} F(\vartheta^{(k_1,i_1)}, \ldots, \vartheta^{(k_N,i_N)}) w^{(k_1,i_1)} \cdots w^{(k_N,i_N)} \\
  &=\sum_{\boldsymbol{i} \in \mathcal{T}(\boldsymbol{q}(\boldsymbol{k}))} F(\vartheta^{(\boldsymbol{k},\boldsymbol{i})}) w^{(\boldsymbol{k},\boldsymbol{i})},
  \end{align}
where $\vartheta^{(\boldsymbol{k},\boldsymbol{i})} \coloneqq (\vartheta^{(k_1,i_1)}, \ldots, \vartheta^{(k_N,i_N)}) \in \varXi^N$ and
$w^{(\boldsymbol{k},\boldsymbol{i})} \coloneqq \prod_{n=1}^N w^{(k_n,i_n)} \in \mathbf{R}$ denote the tensorized quadrature nodes and weights, respectively.
It is obvious from \eqref{eq:ftcard} that for a high dimension $N$, the full tensor quadratures are computationally feasible only if the order $\boldsymbol{k}$ is sparse, i.e., consists mostly of zeros.
The full tensor pseudospectral projection $\mathscr{P}_{\boldsymbol{k}} \colon C(\varXi^N) \to \mathbf{P}_{\boldsymbol{p}(\boldsymbol{k})}$ of order $\boldsymbol{k} \in \mathbf{N}_0^P$ is defined as
  \begin{align}
  \mathscr{P}_{\boldsymbol{k}}(F)(\vartheta) :\!\!&= (\mathscr{P}_{k_1} \otimes \cdots \otimes \mathscr{P}_{k_N})(F)(\vartheta) \\
  &= \sum_{\boldsymbol{i} \in \mathcal{T}(\boldsymbol{p}(\boldsymbol{k}))} \widehat{F}_{\boldsymbol{i}} L_{\boldsymbol{i}}(\vartheta),
  \end{align}
where the pseudospectral coefficients are
  \begin{equation}
  \widehat{F}_{\boldsymbol{i}} \coloneqq \mathscr{Q}_{\boldsymbol{k}}(F L_{\boldsymbol{i}}).
  \end{equation}
The tensorized difference projection operator $\mathscr{D}_{\boldsymbol{k}} \colon C(\varXi^N) \to \mathbf{P}_{\boldsymbol{p}(\boldsymbol{k})}$ is unsurprisingly
  \begin{equation}
  \label{eq:tensordiff}
  \mathscr{D}_{\boldsymbol{k}} \coloneqq \mathscr{D}_{k_1} \otimes \cdots \otimes \mathscr{D}_{k_P} =
  \mathscr{P}_{\boldsymbol{k}} + \sum_{\boldsymbol{i} \in \mathcal{B}(\boldsymbol{k})} (-1)^{\lVert \boldsymbol{k}-\boldsymbol{i} \rVert_1} \mathscr{P}_{\boldsymbol{i}},
  \end{equation}
where
  \begin{equation}
  \label{eq:backward}
  \mathcal{B}(\boldsymbol{k}) \coloneqq \{ \boldsymbol{i} \in \mathbf{N}_0^N : \lVert \boldsymbol{k} - \boldsymbol{i} \rVert_\infty = 1, \; i_n \leq k_n,\; 1 \leq n \leq N \}
  \end{equation}
is (a part of) the \emph{backward neighborhood} of the multi-index $\boldsymbol{k} \in \mathbf{N}_0^N$ (see also \cite[Proposition 1.7]{kaarnioja13}).
Note that $0 \leq \lvert \mathcal{B}(\boldsymbol{k}) \rvert \leq 2^N-1$.

In the univariate case, it is easy to see that
  \begin{equation}
  \label{eq:telescope}
  \mathscr{P}_k(F) = \sum_{i=0}^k \mathscr{D}_i(F)
  \end{equation}
due to the telescopic property of the sum.
On the other hand, it is well-known that the approximation $\mathscr{P}_k(F) \approx F$ is good if $F$ is smooth enough and $k$ is sufficiently large; see, e.g., \cite{CHQZ2} for proofs and details.
The sparse pseudospectral approximation generalizes the idea of \eqref{eq:telescope} to higher dimensions via the Smolyak's construction \cite{kaarnioja13}
  \begin{equation}
  \label{eq:smolyak}
  F \approx F^{(\mathcal{K})} \coloneqq \sum_{\boldsymbol{i} \in \mathcal{K}} \mathscr{D}_{\boldsymbol{i}}(F),
  \end{equation}
where $\mathcal{K} \subset \mathbf{N}_0^N$ is (usually) a sparse set, meaning that each multi-index $\boldsymbol{i} \in \mathcal{K}$ contains many zeros so that the workloads of the quadratures entangled in \eqref{eq:tensordiff} stay feasible.
It follows from \eqref{eq:tensordiff} that $F^{(\mathcal{K})}$ is a nontrivial linear combination of the full tensor projections $\{ \mathscr{P}_{\boldsymbol{i}} \}_{\boldsymbol{i} \in \mathcal{K}}$.
Naturally, the projections can further be written as linear combinations of Legendre polynomials.
The set of Legendre degrees in \eqref{eq:smolyak} is denoted by $\mathcal{P}(\mathcal{K})$ which (implicitly) depends on the chosen quadrature.
For the Gauss--Legendre quadrature we have $\mathcal{P}(\mathcal{K}) = \mathcal{K}$.

The approximation \eqref{eq:smolyak} is valid only if the index set $\mathcal{K}$ is \emph{Smolyak admissible}, which means that every $\boldsymbol{i} \in \mathcal{K}$ satisfies
  \begin{equation}
  \label{eq:adm}
  \text{if} \quad i_n > 0, \quad \text{then} \quad \boldsymbol{i}-\boldsymbol{e}_n \in \mathcal{K}
  \end{equation}
for each $n=1, \ldots, N$.
This makes the sum in \eqref{eq:smolyak} telescope-like.
Clearly, full tensor index sets are Smolyak admissible, but a full tensor index set being sparse means that the approximation is poor for most of the dimensions.
A common alternative to the full tensor index set is the total order index set, which for $k \in \mathbf{N}_0$ is
  \begin{equation}
  \label{eq:totalorder}
  \mathcal{T}_{\text{tot}}(k) \coloneqq \Big\{ \boldsymbol{i} \in \mathbf{N}_0^N : \sum_{n=1}^N i_n \leq k \Big\}.
  \end{equation}
The cardinality of the total order index set is
  \begin{equation}
  \label{eq:tdcard}
  \lvert \mathcal{T}_{\text{tot}}(k) \rvert = \binom{N+k}{k} \sim N^k,
  \end{equation}
which for a large dimension $N$ is significantly less than the cardinality $(k+1)^N$ of the full tensor index set based on a constant vector $(k,\ldots,k) \in \mathbf{N}_0^N$.

\subsection{Adaptive SPAM}
\label{ssec:adaptive}

The aim of the adaptive SPAM is to construct a sparse index set $\mathcal{K}$ in \eqref{eq:smolyak} so that the approximation is good while the cardinality of the set is kept as small as possible.
In many applications, especially when the target function is expensive to evaluate, the main trouble with a too large index set is the huge number of quadrature nodes, since the target function has to be evaluated at all of them.
In our application, however, there is also another incentive for having a small index set $\mathcal{K}$ and a small polynomial set $\mathcal{P}(\mathcal{K})$, namely, the evaluation cost of the surrogate and its derivatives when the inverse problem of thermal tomography is solved with an iterative method.

From now on, we consider the temperature measurement vector $U \colon \varXi^N \to \R^M$ instead of a generic real-valued function $F$, and thus we construct a polynomial approximation $U^{(\mathcal{K})}$ for $U$ as in \eqref{eq:surrogate}.
The quadrature and projection operators directly generalize to the vector-valued case.
As mentioned in the beginning of this section, it is also possible to construct the sets of basis polynomials individually for each component, but for the sake of generality we present the adaptive algorithm for the vector-valued function $U$.

The algorithm for constructing $\mathcal{K}$ is essentially the same as the adaptive quadrature presented in \cite{gerstner03} and goes as follows.
At the beginning, we choose $\mathcal{K} = \{ \boldsymbol{0} \} \subset \mathbf{N}_0^N$.
Then, at each iteration, we select one not-yet-selected {\em critical multi-index} $\boldsymbol{k} \in \mathcal{K}$ according to the selection rule specified shortly.
All the \emph{admissible forward neighbors}
  \begin{equation}
  \mathcal{K}^+(\boldsymbol{k}) \coloneqq \{ \boldsymbol{i} \in \mathbf{N}_0^N : \boldsymbol{i} = \boldsymbol{k}+\boldsymbol{e}_j,\; 1 \leq j \leq N,\; \mathcal{K} \cup \{ \boldsymbol{i} \} \text{ is Smolyak admissible} \}
  \end{equation}
of the selected critical index are then added to $\mathcal{K}$ and this selection-addition procedure is continued until some suitable stopping criterion is satisfied.
The determination of the critical indices encompasses pseudospectral projections so that when the algorithm terminates, we have not only computed the index set $\mathcal{K}$ but also all the necessary information that is needed when evaluating the polynomial $U^{(\mathcal{K})}$ for any given parameter vector $\vartheta \in \varXi^N$.
The expansion coefficients in \eqref{eq:surrogate} correspond to a matrix $\alpha \in \mathbf{R}^{M \times \lvert \mathcal{P}(\mathcal{K}) \rvert}$ and the polynomial degrees can be stored in a sparse matrix $\mathcal{P}(\mathcal{K}) \in \mathbf{N}_0^{\lvert \mathcal{P}(\mathcal{K}) \rvert \times N}$.

There are different selection criteria for the critical multi-index, but typically they are based on the norms of the difference projection operators.
In this paper, we simply consider the norms
  \begin{equation}
  \epsilon(\boldsymbol{k}) \coloneqq \lVert \mathscr{D}_{\boldsymbol{k}}(U) \rVert_{[L^2(\varXi^N)]^M} = \left( \sum_{m=1}^M \lVert \mathscr{D}_{\boldsymbol{k}}(U_m) \rVert_{L^2(\varXi^N)}^2 \right)^{1/2},
  \end{equation}
where
  \begin{equation}
  \label{eq:DkUm}
  \mathscr{D}_{\boldsymbol{k}}(U_m) \, \eqqcolon \sum_{\boldsymbol{i} \in \mathcal{T}(\boldsymbol{p}(\boldsymbol{k}))} \beta_{m,\boldsymbol{i}}^{(\boldsymbol{k})} L_{\boldsymbol{i}}
  \end{equation}
for some coefficients $\beta_{m,\boldsymbol{i}}^{(\boldsymbol{k})} \in \mathbf{R}$ that can be obtained by expanding the expression in \eqref{eq:tensordiff}.
Among the not-yet-selected multi-indices, the one corresponding to the largest $\epsilon$-norm has (loosely speaking) affected the approximation most and thus can be considered the critical index, whose admissible forward neighbors are added to the index set.
By using the orthonormality of the Legendre polynomials, we can write
  \begin{equation}
  \lVert \mathscr{D}_{\boldsymbol{k}}(U_m) \rVert_{L^2(\varXi^N)} = \left( \sum_{\boldsymbol{i} \in \mathcal{T}(\boldsymbol{p}(\boldsymbol{k}))} \left(\beta_{m,\boldsymbol{i}}^{(\boldsymbol{k})}\right)^2 \right)^{1/2},
  \end{equation}
so that
  \begin{equation}
  \epsilon(\boldsymbol{k}) = \lVert \beta^{(\boldsymbol{k})} \rVert_F,
  \end{equation}
which is the Frobenius norm of the coefficient matrix in \eqref{eq:DkUm}.
According to the Smolyak's construction \eqref{eq:smolyak}, accumulating the coefficients $\beta^{(\boldsymbol{k})}$ computed during the algorithm results in the approximation $U^{(\mathcal{K})}$, and the polynomial degrees $\mathcal{P}(\mathcal{K})$ in \eqref{eq:surrogate} are just the union of the degrees appearing in the difference projections of the form \eqref{eq:DkUm}.

\begin{algorithm}[b!]
  \label{alg:adaptive}
  \SetKwInOut{Input}{Input}
  \SetKwInOut{Output}{Output}

  \underline{function adaptiveSPAM} $(U)$\;
  \Input{Function $U \in [C(\varXi^N)]^M$}
  \Output{Polynomial $U^{(\mathcal{K})} \approx U$ of the form \eqref{eq:surrogate}}

  $U^{(\mathcal{K})} \gets 0$\tcc*[r]{zero polynomial to start with}
  $\mathcal{K} \gets \emptyset$\tcc*[r]{indices already projected}
  $\mathcal{J} \gets \{\boldsymbol{0}\}$\tcc*[r]{indices to be projected next}

  \While{$\lvert \mathcal{P}(\mathcal{K}) \rvert < \text{stopping criterion}$}
    {
    \For{$\boldsymbol{j} \in \mathcal{J}$}
      {
      Compute and store $\mathscr{P}_{\boldsymbol{j}}(U) \in \mathbf{P}_{\boldsymbol{p}(\boldsymbol{j})}$\tcc*[r]{forward solver needed}
      $\mathscr{D}_{\boldsymbol{j}}(U) \gets \mathscr{P}_{\boldsymbol{j}}(U) +
      \sum_{\boldsymbol{i} \in \mathcal{B}(\boldsymbol{j})} (-1)^{\lVert \boldsymbol{j}-\boldsymbol{i} \rVert_1} \mathscr{P}_{\boldsymbol{i}}(U)$\tcc*[r]{nothing to solve}
      Compute and store $\epsilon(\boldsymbol{j})$\;
      $U^{(\mathcal{K})} \gets U^{(\mathcal{K})} + \mathscr{D}_{\boldsymbol{j}}(U)$\tcc*[r]{accumulate output ($\mathcal{P}(\mathcal{K})$ and $\alpha$)}
      }
    $\mathcal{K} \gets \mathcal{K} \cup \mathcal{J}$\;
    $\boldsymbol{k} \gets \argmax_{\boldsymbol{i} \in \mathcal{K}} \epsilon(\boldsymbol{i})$\tcc*[r]{new critical index}
    $\epsilon(\boldsymbol{k}) \gets -\infty$\tcc*[r]{do not select again}
    $\mathcal{J} \gets \mathcal{K}^+(\boldsymbol{k})$\tcc*[r]{admissible forward neighbors}
    }
  \caption{Adaptive sparse pseudospectral approximation}
\end{algorithm}

In addition to the $\epsilon$-norm introduced above, one may want to weigh different multi-index candidates based on, e.g., how much it costs to compute the associated projections or how many new polynomials one is going to obtain.
See, e.g., \cite{conrad13} for work-considering algorithms.
For Gauss--Legendre quadrature, adding one multi-index $\boldsymbol{k}$ corresponds to adding precisely one new polynomial, namely $L_{\boldsymbol{k}}$.
In consequence, as we exclusively resort to Gauss--Legendre rule in our numerical experiments, we employ $\epsilon$ as the selection criterion.

The pseudocode of the method is presented in Algorithm~\ref{alg:adaptive}.
Note that $0 \leq \lvert \mathcal{K}^+(\boldsymbol{k}) \rvert \leq N$ and the Smolyak admissibility implies that $\mathcal{B}(\boldsymbol{i}) \subset \mathcal{K}$ for each $\boldsymbol{i} \in \mathcal{K}^+(\boldsymbol{k})$, see \eqref{eq:backward} and \eqref{eq:adm}.
Thus, all the projections appearing on line 8 in Algorithm~\ref{alg:adaptive} have already been computed and the only line where the numerical forward solver for $U$ is needed is line 7.
Recall that by a forward solver we mean a computational function that returns the boundary measurement vector $U(\vartheta) \in \mathbf{R}^M$ for a given parameter vector $\vartheta \in \varXi^N$.
In practice, such a solver only returns numerical approximations.

\section{Numerical experiments}
\label{sec:num}

In this section, we present numerical examples demonstrating the feasibility of the proposed two-phase approach to thermal tomography:
\begin{enumerate}
\item In the offline phase, a polynomial surrogate is formed based (only) on generic information about the measurement setup, including approximate shape and size of the imaged body, the expected mean levels of the to-be-reconstructed parameters, and the number and size of the heaters and sensors. Because such computations can be performed prior to taking any measurements, one can assume the availability of considerable amounts of computation time and power for completing these tasks.
\item Once the measurements are in hand, the online phase consists merely of solving a least squares minimization problem that does not involve the forward solver but only its polynomial surrogate. Since the surrogate and its derivatives can be efficiently evaluated, this leads to fast reconstructions.
\end{enumerate}
We first study the accuracy of the polynomial surrogate in Section~\ref{ssec:surrtest}.
More precisely, the values returned by the surrogate are compared with the values returned by the numerical solver upon which the surrogate is based.
Naturally, the numerical solver itself returns only approximate measurement vectors for given parameter vectors, but since the surrogate can be constructed in the offline phase, one could in principle choose as accurate solver as desired without affecting the efficiency of the actual reconstruction algorithm in the online phase.
The errors in spatial and temporal discretizations are, however, taken into account when choosing the regularization or the Bayesian distributions for the inverse problem in Section~\ref{ssec:invprob}. On the other hand, since the number of employed basis polynomials affects the cost of evaluating the surrogate and its derivatives in the online phase \cite{mustonen15}, one should exploit adaptivity to keep the number of polynomials as low as possible without compromising the accuracy of the surrogate.

Section~\ref{ssec:invprob} addresses the inverse problem of thermal tomography by comparing some target functions $a$, $b$, $c$ and $\partial \varOmega$ with the corresponding reconstructions; these considerations correspond to the online phase.
The reconstructions are based on simulated measurement data, which are generated by a highly accurate numerical solver and further contaminated by a small amount of artificial noise in order to avoid any kind of an \emph{inverse crime}.
We also examine how reconstructions are affected if the uncertainties in the coefficient $c$ and the boundary $\partial \varOmega$ are not taken into account.

Table~\ref{tab:parvalues} lists the values defining the parametrized measurement setup used in all examples that follow.
In particular, the zero vector $\vartheta=0$ corresponds to the unit disk with constant fields $a \equiv b \equiv 0.55$ and the transfer coefficient $c = 0.11$ between the heaters.
Both the heater starting edges and the sensors are positioned equiangledly, and when the domain is the unit disk, the sensors are located exactly at the midpoints between the heaters.
The geometry is illustrated in Figure~\ref{fig:setup}; see also Section~\ref{sec:tomo} for more details on the parametrization.
The total number of parameters is either $N=104$ or $N=984$ depending on the discretization of the fields $a$ and $b$.
The number of measurements is always $M=JRM_T=384$.

\begin{table}
  \caption{Parameter values for the numerical experiments.}
  \begin{center}
  \begin{tabular}{c | c | l}
  Parameter & Value & Notes \bigstrut[b] \\
  \hline
  $N_a = N_b$ & $40 \text{ or } 480$ & Discretization of the conductivity $a$ and the capacity $b$ \bigstrut[t] \\
  $J = R = N_c$ & $8$ & Number of heaters, sensors and discretization of the coefficient $c$ \\
  $N_\varOmega$ & $16$ & Discretization of the boundary \\
  $\bar{a} = \bar{b}$ & $0.55$ & \multirow{2}{*}{$\left. \vphantom{\begin{tabular}{c}3\\3\end{tabular}}\right\}$ $0.1 \leq a,b \leq 1$} \\
  $\widetilde{a} = \widetilde{b}$ & $0.45$ & \\
  $\bar{c}$ & $0.11$ & \multirow{2}{*}{$\left. \vphantom{\begin{tabular}{c}3\\3\end{tabular}}\right\}$ $0.02 \leq c \leq 0.2$ between the heaters} \\
  $\widetilde{c}$ & $0.09$ & \\
  $c_0$ & $10$ & $c = 10$ on the heaters \\
  $\rho_0$ & $1$ & Radius of the reference domain $\varOmega(0)$ \\
  $\rho_-$ & $0.8$ & Minimum radius \\
  $\rho_+$ & $1.2$ & Maximum radius \\
  $T$ & $2$ & Maximum time \\
  $M_T$ & $6$ & Number of measurements in time \\
  $t_i$ & $ iT/M_T$ & Uniform measurement times ($i=1,\ldots,M_T$)\\
  $g(t)$ & $5t$ & Time profile of the active heater \\
  $\eta$ & $\pi/8$ & Heater width \bigstrut[b] \\
  \hline
  \end{tabular}
  \end{center}
  \label{tab:parvalues}
\end{table}

\begin{figure}[t]
  \centering
  \includegraphics{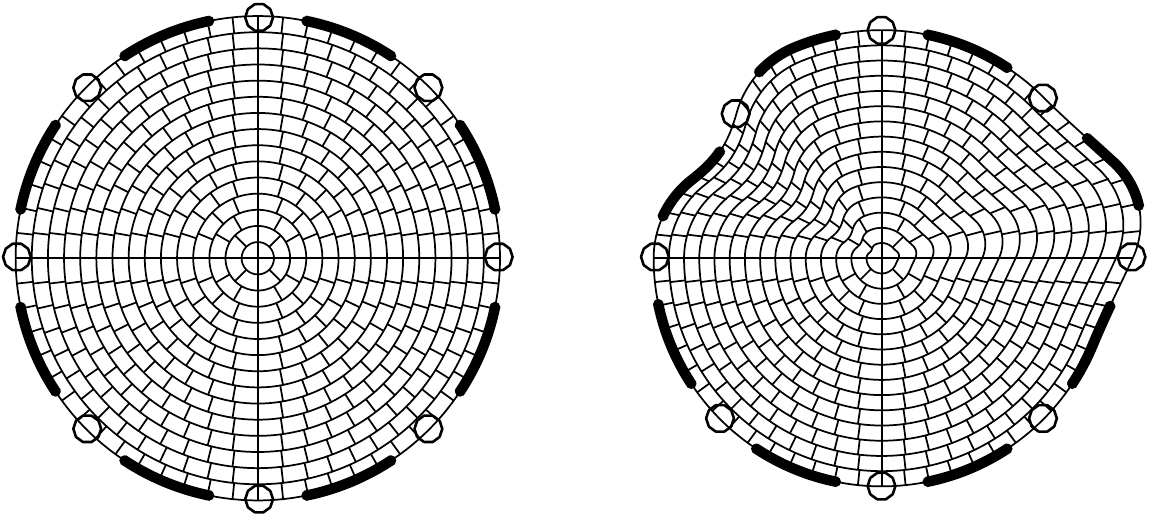}
  \caption{Parametrization of the measurement setup for the numerical experiments. The heaters are indicated by the bold boundary segments and the circles denote the sensors. The discretization of the conductivity and the capacity is shown for $N_a = N_b = 480$. \emph{Left:} Unit disk $\varOmega(0)$. \emph{Right:} Domain that is obtained by perturbing two out of $N_\varOmega=16$ splines.}
  \label{fig:setup}
\end{figure}

\subsection{Accuracy of the surrogate}
\label{ssec:surrtest}

We first focus on the accuracy of the polynomial surrogate which is constructed in three different ways.
The parametrization of the measurement setup is such that $N_a = N_b = 40$ and thus $N=104$ (this coarse discretization is visible in Figures \ref{fig:reco104c} and \ref{fig:reco104bnd}).
The first surrogate $U^{(\mathcal{K})} \colon \varXi^N \to \mathbf{R}^M$ is computed by choosing $\mathcal{P}(\mathcal{K}) = \mathcal{K} = \mathcal{T}_{\text{tot}}(2)$, see \eqref{eq:totalorder}.
That is, adaptivity is not used, but instead we fix the polynomials to those resulting from all full tensor projections having total order at most~$2$.
(Recall that for Gauss--Legendre quadrature it holds that $\mathcal{P}(\mathcal{K}) = \mathcal{K}$, i.e., the polynomial degrees are precisely the projection orders.)
The number of polynomials is $\lvert \mathcal{P}(\mathcal{K}) \rvert = 5565$, which follows from \eqref{eq:tdcard}.
For the second surrogate, we apply the adaptive SPAM as described in Section~\ref{ssec:adaptive}.
The algorithm is run until $\lvert \mathcal{P}(\mathcal{K}) \rvert \geq 5565$ and the same polynomial basis $\mathcal{P}(\mathcal{K})$ is constructed for all measurement components.
The third surrogate is also constructed adaptively, but now separately for each measurement component, so that there are $M$ different polynomial bases, each having cardinality of about $5565$.
In order to avoid too costly computations, the maximum degree for \emph{univariate} Legendre polynomials is set to~$4$ in the third case.

The accuracy of these three surrogates is tested by drawing independent random parameter vectors $\{\vartheta^{(q)}\}_{q=1}^Q \subset \varXi^N$ from two different distributions.
In the first case, all components of the random vector are independently and uniformly distributed on $\varXi = [-1/2,1/2]$.
The second distribution corresponds to (truncated) log-normal random fields for the conductivity $a$ and the heat capacity $b$.
More precisely, $a$ and $b$ are considered as discrete log-normal random fields with the (common) underlying normal distribution having the mean $\log(0.55)$ and the covariance function
  \begin{equation}
  \label{eq:normcovar}
  K(x,y) = \varsigma^2 \exp\mathopen{}\left( - \frac{ \lVert x-y \rVert_2^2 }{2\lambda^2} \right)\mathclose{}
  \end{equation}
with $\varsigma = 0.5$ and $\lambda = 1/3$.
In other words, $K(x,y)$ can be interpreted as the covariance matrix of an $N_a = N_b$ dimensional normal distribution when $x,y \in \varOmega$ run through the centers of the `pixels' corresponding to the characteristic functions $\varphi_i^{(a)} = \varphi_i^{(b)}$ in the unit disk; cf.~\eqref{eq:abparam} and Figure~\ref{fig:setup}.
The fields $a$ and $b$ are drawn mutually independently as are the remaining $N_c + N_\varOmega$ components of $\vartheta^{(q)}$, which again follow the uniform distribution. If a random draw for either of the two fields results in a pixel-value that corresponds to a component $\vartheta_i^{(q)} < -1/2$ (resp., $\vartheta_i^{(q)} > 1/2$), the field is truncated by redefining $\vartheta_i^{(q)} =-1/2$ (resp., $\vartheta_i^{(q)} = 1/2$).

For each surrogate $U^{(\mathcal{K})}$, we define the pointwise error, the mean error and the (sample) variance as
  \begin{equation}
  e_{\mathcal{K}}^{(q)} \coloneqq \big\lVert U(\vartheta^{(q)}) - U^{(\mathcal{K})}(\vartheta^{(q)}) \big\rVert_2, \qquad
  \mu_{\mathcal{K}} \coloneqq \frac{1}{Q} \sum_{q=1}^Q e_{\mathcal{K}}^{(q)} \quad \text{and} \quad
  \sigma_{\mathcal{K}}^2 \coloneqq \frac{1}{Q-1} \sum_{q=1}^Q ( e_{\mathcal{K}}^{(q)} - \mu_{\mathcal{K}} )^2,
  \end{equation}
respectively.
Here, $U$ denotes our numerical forward solver that is based on piecewise linear finite elements with meshes having a couple of thousands of nodes and the Crank--Nicolson time step of size $1/30$.
We emphasize that the same solver is used for both the surrogate construction and the computation of the reference samples $\{U(\vartheta^{(q)})\}_{q=1}^Q$ since the aim at this point is to examine the accuracy of the polynomial approximation, not the accuracy of the numerical forward solver itself.

\begin{table}
  \caption{Mean errors and variances for two random parameter distributions and three surrogates $U^{(\mathcal{K})}$: Total order polynomials of degree 2 (T2), adaptively constructed polynomials with a common basis (adapt.) and adaptively constructed polynomials with separate bases ($m$-adapt.).}
  \begin{center}
  \begin{tabular}{l | c c c | c c c}
  & \multicolumn{3}{c|}{Mean ($\mu_{\mathcal{K}} \cdot 10^{2}$)} & \multicolumn{3}{c}{Variance ($\sigma^2_{\mathcal{K}} \cdot 10^{2}$)} \\
  & T2 & adapt. & $m$-adapt. & T2 & adapt. & $m$-adapt. \bigstrut[b] \\
  \hline
  Uniform & $58.3$ & $20.2$ & $9.08$ & $11.3$ & $1.13$ & $0.135$ \bigstrut[t] \\
  Log-normal & $50.2$ & $17.2$ & $7.58$ & $6.08$ & $0.775$ & $0.127$ \bigstrut[b] \\
  \hline
  \end{tabular}
  \end{center}
  \label{tab:errors}
\end{table}

Table~\ref{tab:errors} lists the mean errors and variances corresponding to the two distributions for the parameter vector described above with sample size $Q=1000$. It is evident that adaptivity improves the accuracy of the surrogate: the mean errors for the (commonly) adaptively constructed surrogate are about one third and the sample variances about one tenth of the corresponding values for the total order surrogate with $\mathcal{P}(\mathcal{K}) = \mathcal{T}_{\text{tot}}(2)$. Because the adaptively constructed basis has approximately the same number of polynomials as the total order basis, the evaluation and differentiation costs of these two surrogates are expected to be roughly the same.
Although the complexity analysis in \cite{mustonen15} considers only total order bases, in practice there seems to be no significant difference in computational work between these two types of surrogates.
In particular, for a given parameter vector, the evaluation of the surrogate still reduces to a simple matrix-vector product even if adaptivity is used.
We refer to \cite{mustonen15} for more detailed discussion on computational work involved when using the surrogate, and also remind the reader that in any case the cost of evaluating and differentiating the surrogate does not depend on the discretization of the forward solver that is used to form the surrogate.
What is more, the adaptivity does not essentially slow down the offline phase, i.e., the construction of the surrogate, albeit some more evaluations of the forward solver may be needed.

Constructing the polynomial basis separately for each measurement component further enhances the surrogate, as seen in Table~\ref{tab:errors}: the mean errors are only about $15\%$ of the corresponding numbers for $\mathcal{P}(\mathcal{K}) = \mathcal{T}_{\text{tot}}(2)$. However, in this case the surrogate construction algorithm has to be run several times, and the resulting surrogate cannot be evaluated as a simple matrix-vector product as in the case of a common basis (cf.~\cite{mustonen15}).

When the random realizations for $a$ and $b$ are smoother (i.e., the values of adjacent `pixels' are expected to be strongly correlated) as they are in the log-normal case, the mean errors seem to be slightly smaller compared to the case of random parameter vectors with independently and uniformly distributed components.
For the zero parameter vector $\vartheta = 0$, the errors for the total order, adaptively constructed and adaptively/separately constructed surrogates are $e_{\mathcal{K}}^{(q)}(0) = 9.60\cdot 10^{-2}$, $4.26\cdot 10^{-2}$ and $1.41\cdot 10^{-2}$, respectively. It is worth noticing that the mean errors for the surrogate that is constructed adaptively and separately for each measurement component are smaller than the pointwise error of the total order surrogate at the origin $\vartheta=0$ (cf.~Table~\ref{tab:errors}).

\begin{figure}[t]
  \centering
  \includegraphics{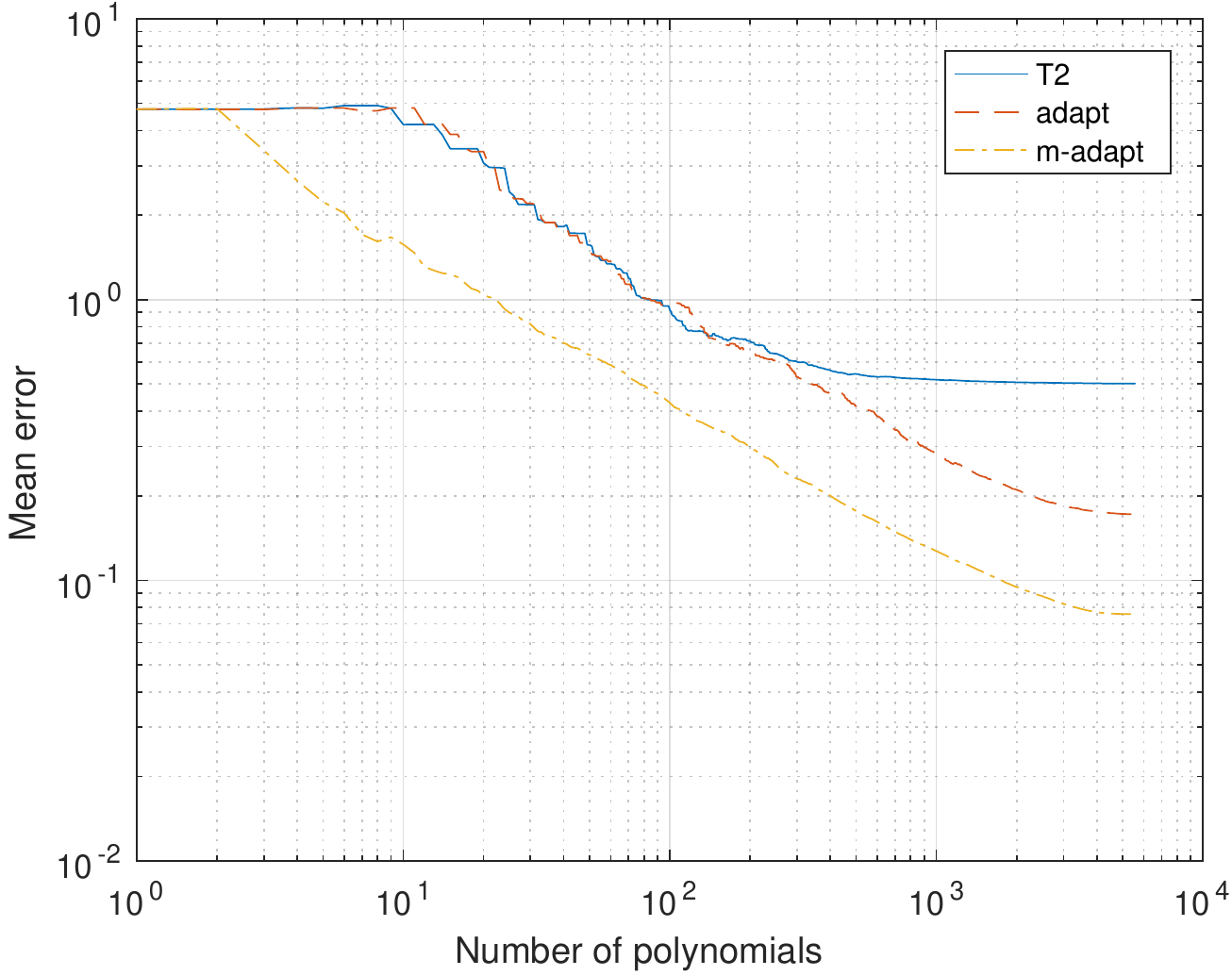}
  \caption{Mean errors for three different surrogates as functions of the expansion size.}
  \label{fig:conv}
\end{figure}

Figure~\ref{fig:conv} shows the mean errors for the log-normal random realizations if only a subset of the computed polynomials is used in each surrogate.
More precisely, for each surrogate, the polynomials are first sorted based on the norms of the coefficient vectors $(\alpha_{m,\boldsymbol{i}})_{m=1}^M$ in the expansion~\eqref{eq:surrogate}.
Then the polynomials having the largest magnitudes for their coefficients are retained while the others are discarded.
For the separately constructed surrogate this filtering is performed component-wise, so that the norms above are just the absolute values.
The values at the right end of the convergence plots in Figure~\ref{fig:conv} are those on the bottom line in Table~\ref{tab:errors}.
It is clear that even a prematurely stopped adaptive algorithm, say, at $500$ polynomials, overwhelms the total order surrogate, although the values in Figure~\ref{fig:conv} do not exactly correspond to the mean errors one would obtain by stopping the algorithm after obtaining a certain number of polynomials. For the adaptively and separately formed surrogate this effect is even more apparent: already $100$ polynomials per measurement component is enough for achieving higher accuracy than the total order surrogate with $\mathcal{P}(\mathcal{K}) = \mathcal{T}_{\text{tot}}(2)$.

The error between our standard numerical forward solver for the surrogate construction and a highly accurate solver is $1.62\cdot 10^{-2}$ for the case $\vartheta=0$, measured in the same way as the error between the standard numerical solver and the corresponding surrogates.
The highly accurate solver has tens of thousands of FEM nodes and a Crank--Nicolson time step of $1/50$.
The discrepancy between the two forward solvers stems mainly from the spatial discretization, and
it is expected to be higher for more complicated domains and nonconstant fields $a$ and $b$.
Roughly speaking, we anticipate that the discretization error of our standard numerical forward solver is of the same order as the surrogate error for the most accurate of the three surrogates.

\subsection{Inverse problem}
\label{ssec:invprob}

For the inverse examples that follow, we generate the (noiseless) boundary measurement data $\bar{U} \in \R^M$ by solving the problem \eqref{eq:weak} with the highly accurate solver described at the end of the previous section.
In general, the coefficient functions and boundary shapes used for simulating the data are not (exactly) representable by the parametrization upon which the surrogates are built. However, the circumference for all target objects is chosen to be $2 \pi = |\partial \varOmega(0)|$; it is not too far-fetched to assume that the approximate size of the imaged object is known in practice. A realization of a vector of zero mean Gaussian noise with independent components is added to the measurement, with the standard deviation of the $m$th component being $5 \cdot 10^{-3} \cdot \bar{U}_m$. The resulting noisy measurement vector is denoted by $\widetilde{U} \in \R^M$.

The inverse problem is treated as a nonlinear least squares minimization
  \begin{equation}
  \label{eq:lsq}
  \argmin_{\vartheta \in \varXi^N}  \left\{ \big\lVert \widetilde{U} - U^{(\mathcal{K})}(\vartheta) \big\rVert_2^2 + \delta^2 \lVert G \vartheta \rVert_2^2 \right\},
  \end{equation}
where $G \in \mathbf{R}^{N \times N}$ is a block-diagonal regularization matrix.
Its first two blocks of lengths $N_a = N_b$ are $\chol(K^{-1})$, where $\chol(\, \cdot \,)$ denotes the upper-triangular Cholesky factor of a matrix and $K$ is formed by evaluating the covariance function \eqref{eq:normcovar} with $\varsigma^2 = 0.5$ and $\lambda=1/3$ at all centerpoint pairs for the pixels in $\varOmega(0)$.
Therefore, the regularization operator prefers smooth fields $a$ and $b$ as in,~e.g.,~\cite{darde13b,toivanen14}.
If the sum of the noise in $\widetilde{U}$ and the surrogate and numerical errors in $U^{(\mathcal{K})}$ was assumed to be Gaussian with independent zero-mean components having standard deviation $\delta>0$, then \eqref{eq:lsq} would have a Bayesian interpretation of maximizing the posterior distribution for the parameter $\vartheta$ under the prior assumption that the fields $a$ and $b$ are mutually independent Gaussian random fields with pointwise variance $\varsigma^2$ and correlation length $\lambda$.
To be quite precise, the regularization matrix should be updated based on the current shape of the domain, but since the boundary perturbations are relatively small, it is considered to be enough to compute the corresponding covariances only in the reference disk.
We employ no regularization (i.e.,~no informative prior) for the coefficient $c$ or the shape of the object, meaning that the lower right-hand corner of $G$ is actually empty.
Such regularization would become necessary if the respective numbers of degrees of freedom, $N_c$ and $N_\varOmega$, were increased significantly.

We use the generic value $\delta = 10^{-2}$ for the regularization parameter in all the following examples.
Note that $\delta$ and $\varsigma$ couple in \eqref{eq:lsq} in such a way that decreasing $\delta$ is equivalent to increasing $\varsigma$.
There is, however, a semi-heuristic Bayesian justification for our choices of $\delta$ and $\varsigma$.
Since the values of the functions $a$ and $b$ vary between $0.1$ and $1$, the choice $\varsigma^2 = 0.5$ seems reasonable.
On the other hand, assuming (falsely) that the sums of the artificial noise and the surrogate and discretization errors at the sensors are independent mean-free Gaussians with the same distribution, one can choose $\delta \approx e / \sqrt{M}$, where $e \approx 20 \cdot 10^{-2}$ is a crude approximation for the magnitude of the total error in the Euclidean norm when the surrogate is based on the standard adaptive algorithm (cf.~the second column of Table~\ref{tab:errors}).
Increasing $\delta$ would generally force the minimization algorithm to converge more robustly, but it would also result in reconstructions that are closer to the average values and not so informative.
Moreover, increasing the amount of measurement noise would make the reconstructions less accurate.
The choice of $\lambda=1/3$ reflects our prior assumption on the characteristic length of variations (in $a$ and $b$) inside $\varOmega$.
If $\lambda$ is increased, the reconstructions become more constant, whereas decreasing $\lambda$ allows reproducing finer details but also leads to unwanted fluctuations in reconstructions of areas described by constant parameters.

\begin{figure}[t]
  \centering
  \includegraphics{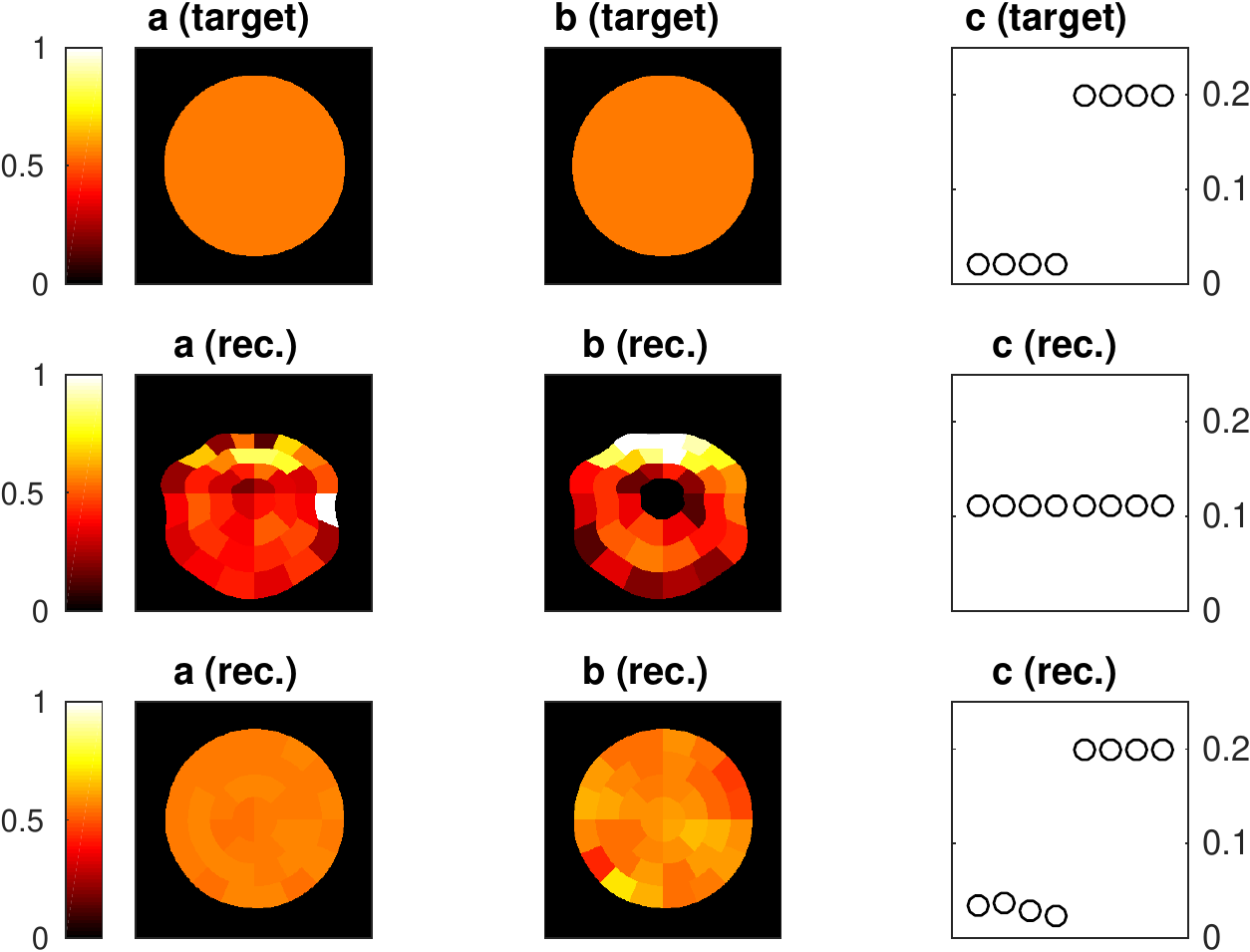}
  \caption{Target configuration (\emph{top}), reconstruction without reconstructing the coefficient $c$ (\emph{middle}) and full reconstruction (\emph{bottom}). $N_a = N_b = 40$.}
  \label{fig:reco104c}
\end{figure}

We solve the problem \eqref{eq:lsq} by using the lsqnonlin function in Matlab without constraints starting from the initial guess $\vartheta=0$. However, any nonlinear least squares approach should work, since computing the values and Jacobians of the to-be-minimized function only requires polynomial evaluations and applying the regularization makes the problem somewhat well-behaving.
Constraints that force the parameter vector to stay within $\varXi^N$ could be implemented if necessary, but in our examples the algorithm seems to work without such.
The reconstructed functions $a$, $b$, $c$ and the shape of $\partial \varOmega$ can be deduced from the minimizing parameter vector $\vartheta$ by considering the mappings introduced in Section~\ref{ssec:param}.
Observe that the parameter values in Table~\ref{tab:parvalues} are quite arbitrary and no big effort has been made to select a measurement configuration that yields maximal amount of information about the target (notice that there is a lot of freedom in the choice of,~e.g.,~the heating pattern $g$).
This task of \emph{optimal experimental design} is not addressed in this paper; see,~e.g.,~\cite{huan13} for more on that topic.

Let us first consider the (computationally) easier case $N_a=N_b=40$ and investigate  what happens if the uncertainties in $c$ and $\partial \varOmega$ are not taken into account.
To this end, two additional `bad' surrogates are constructed.
The first one assumes that $c=\bar{c}=0.11$ between the heaters and so the number of parameters is effectively reduced by $N_c$.
This surrogate is formed adaptively and a common polynomial basis of size $\lvert \mathcal{P}(\mathcal{K}) \rvert \approx 5565$ is used for every measurement.
Figure~\ref{fig:reco104c} shows the target configuration which corresponds to the case $\vartheta=0$ except that the coefficient $c$ takes two distinct values, namely $0.02$ on the first four boundary segments between the heaters and $0.2$ on the remaining ones. The second and third rows of the figure demonstrate that optimizing \eqref{eq:lsq} only for $\vartheta^{(a)}$, $\vartheta^{(b)}$ and $\vartheta^{(\varOmega)}$ employing the new surrogate yields an inaccurate reconstruction, whereas using the full adaptively computed common-basis surrogate from Section~\ref{ssec:surrtest} with $\lvert \mathcal{P}(\mathcal{K}) \rvert \approx 5565$ and optimizing the whole vector $\vartheta$ results in a far superior outcome.

The next target setup is illustrated in the top row of Figure~\ref{fig:reco104bnd}; it corresponds to vanishing subvectors $\vartheta^{(a)}$, $\vartheta^{(b)}$ and $\vartheta^{(c)}$.
The second bad surrogate is formed by assuming that the domain is the unit disk, effectively removing $\vartheta^{(\varOmega)}$ from our parameter vector.
This surrogate is again computed by adaptive, common-basis method and involves $\lvert \mathcal{P}(\mathcal{K}) \rvert \approx 5565$ polynomials. It is then employed in \eqref{eq:lsq} to compute the corresponding reconstruction, which is shown on the second row of Figure~\ref{fig:reco104bnd}. It is obvious that ignoring the uncertainties in the shape of $\varOmega$ heavily deteriorates the reconstructions of the other parameters.
The third row presents the result of using the full adaptive common-basis surrogate in \eqref{eq:lsq}: 
the reconstruction is clearly better than when ignoring $\vartheta^{(\varOmega)}$, albeit not as accurate as that on the third row of Figure~\ref{fig:reco104c} where the same surrogate was used for a different target.

\begin{figure}[t]
  \centering
  \includegraphics{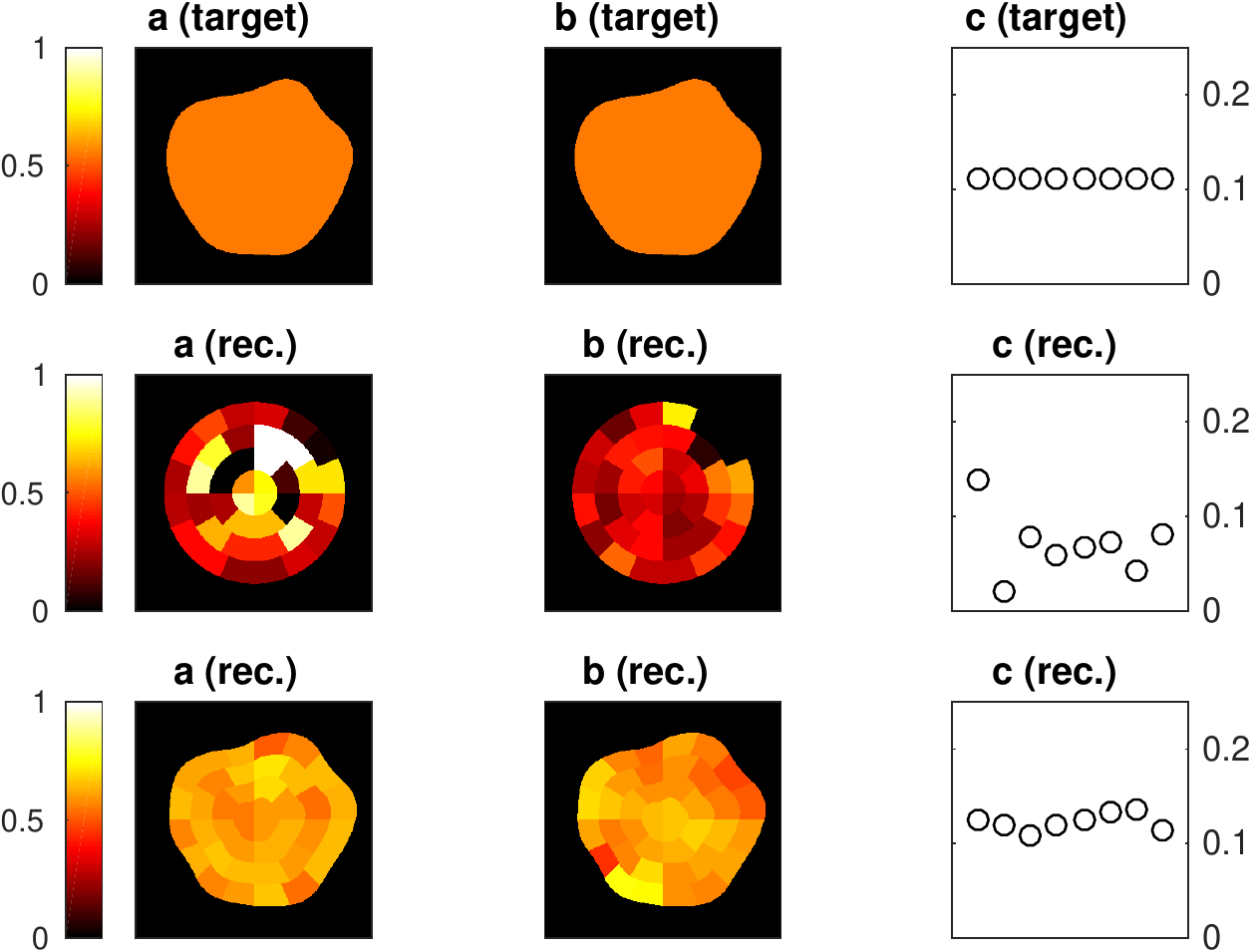}
  \caption{Target configuration (\emph{top}), reconstruction without reconstructing the shape of $\varOmega$ (\emph{middle}) and full reconstruction (\emph{bottom}). $N_a = N_b = 40$. }
  \label{fig:reco104bnd}
\end{figure}

Finally, let us consider a surrogate that corresponds to $N_a = N_b = 480$ as in Figure~\ref{fig:setup}, resulting in $N=984$ variables.
One polynomial basis is adaptively constructed for two adjacent active heaters at a time, so that there are altogether four different bases.
Due to certain symmetries in the parametrization, choosing to use such a four-bases surrogate does not actually increase the amount of offline computations.
The maximal univariate Legendre degree is set to $4$ and the adaptive algorithm is terminated after obtaining $485\,605$ polynomials, which is the same number as in the total order set $\mathcal{T}_{\rm tot}(2)$ for $984$ parameters. The maximal total polynomial degree found by the algorithm is $11$.
About $81\%$ of the polynomials are of the second order.
The higher order polynomials are mostly associated with the shape variables $\vartheta^{(\varOmega)}$ and the parameters in $\vartheta^{(a)}$ defining the thermal conductivity $a$ close to the exterior boundary.
However, the parameter choices in Table~\ref{tab:parvalues} obviously have a significant impact on the way the adaptive algorithm tends to distribute the higher order polynomials.

Figure~\ref{fig:reco984-1} presents the result of the surrogate inversion for a smooth target configuration.
Both the thermal conductivity and the heat capacity are relatively well reconstructed along with the object shape. The same applies to a somewhat lesser extent to the heat transfer coefficient.
Figure~\ref{fig:reco984-2} documents a more demanding reconstruction. In this case, the piecewise constant target functions for $a$ and $b$ are not compatible with the chosen smoothness regularization.
However, the reconstructions of $a$ and $b$ arguably still carry useful qualitative information about the interior structure of the imaged object --- in particular, far more information than one would obtain if the uncertainties in the exterior boundary and the heat transfer coefficient were simply ignored in the inversion.

The online phase of computing the reconstructions in Figures~\ref{fig:reco984-1} and \ref{fig:reco984-2},~i.e.,~solving the minimization problem~\eqref{eq:lsq} with $N=984$ by the lsqnonlin function of Matlab with a user-supplied Jacobian, took only roughly ten seconds on a standard desktop computer.
Obviously, the efficiency depends on the implementation and on the precise settings of the minimization procedure.
Note also that the coefficient matrix $\alpha \in \mathbf{R}^{M \times \lvert \mathcal{P}(\mathcal{K}) \rvert}$ contains almost $2 \cdot 10^8$ elements so storing it requires quite a bit of memory.

\begin{figure}[t]
  \centering
  \includegraphics{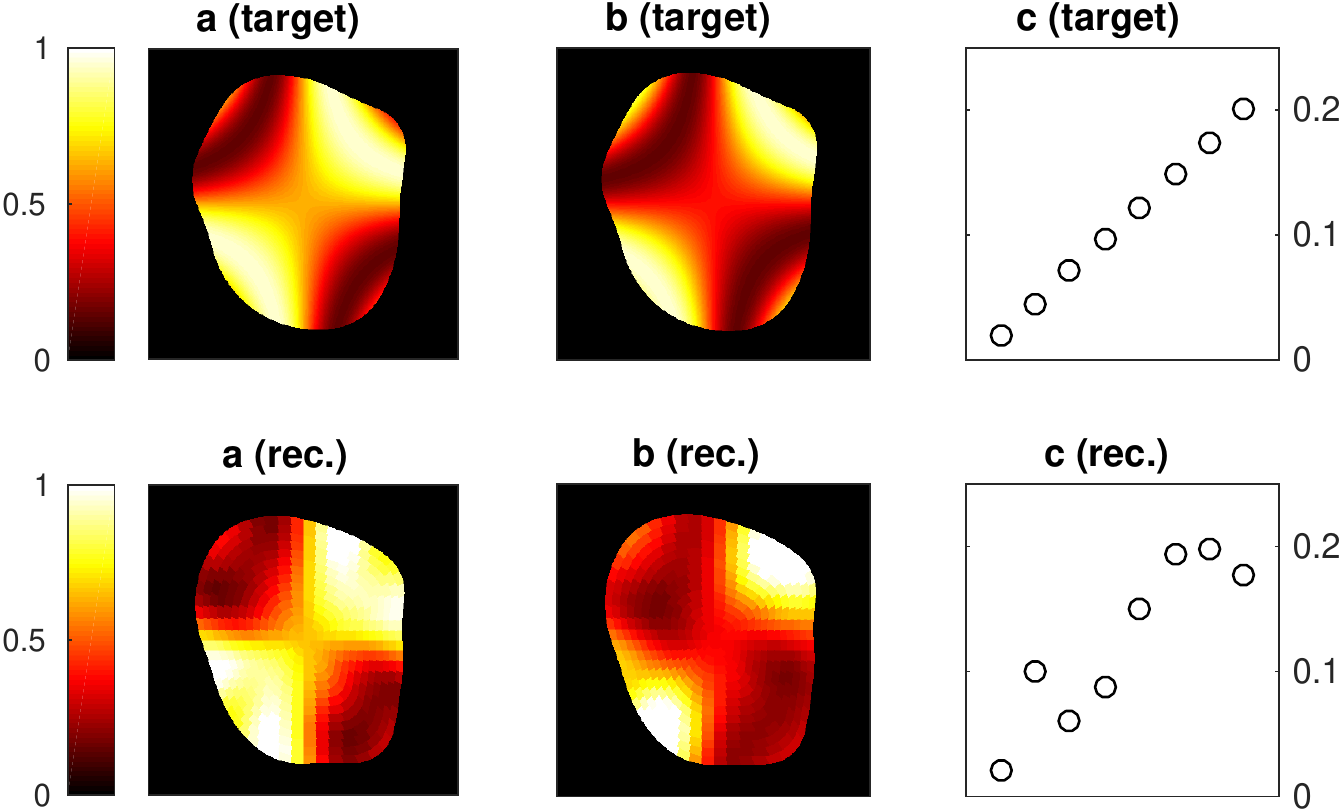}
  \caption{Smooth target configuration (\emph{top}) and its reconstruction using a surrogate with $N = 984$ variables (\emph{bottom}).}
  \label{fig:reco984-1}
\end{figure}

\begin{figure}[t]
  \centering
  \includegraphics{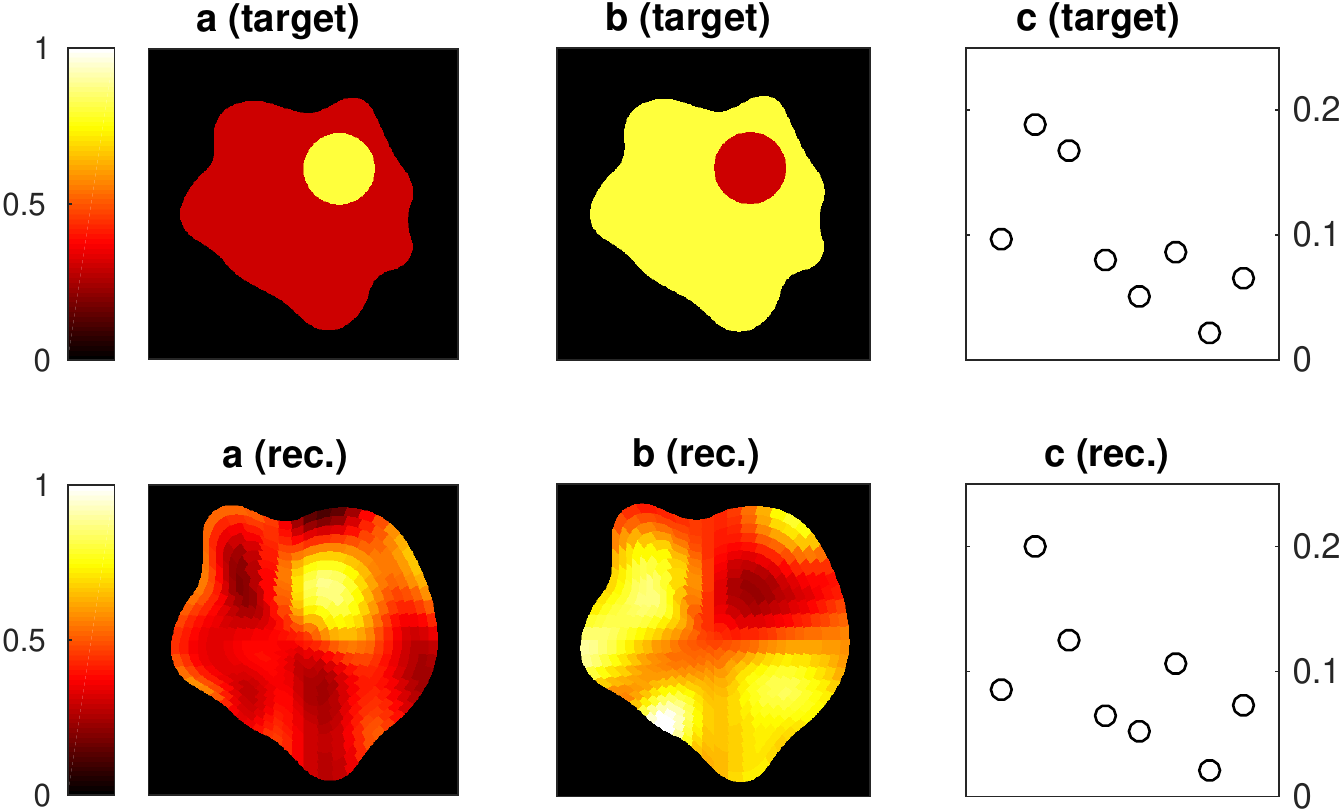}
  \caption{Piecewise constant target functions (\emph{top}) and their reconstructions using a surrogate with $N = 984$ variables (\emph{bottom}).}
  \label{fig:reco984-2}
\end{figure}

\section{Conclusions}
\label{sec:conclusions}

We have extended the previous works on thermal tomography to include estimating the exterior shape of the imaged object in addition to reconstructing the thermal conductivity, heat capacity and surface conductance.
The presented algorithm is based on the adaptive pseudospectral approximation approach and simple output least squares minimization, resulting in a fast inversion method that requires only polynomial evaluation and differentiation while being independent of the discretization of the parabolic forward problem.
Our numerical examples demonstrated that the shape estimation is indeed important if one wants to obtain reasonable reconstructions of the thermal parameters inside a physical domain whose shape is not accurately known.

\section*{Acknowledgments}
This work was supported by the Academy of Finland (decision 267789) and the Finnish Foundation for Technology Promotion TES.

\bibliographystyle{ppde}
\bibliography{ref}

\end{document}